\documentclass{amsart}

\newtheorem{theorem}{Theorem}[section]
\newtheorem{lemma}[theorem]{Lemma}

\theoremstyle{definition}
\newtheorem{definition}[theorem]{Definition}

\theoremstyle{remark}
\newtheorem{remark}[theorem]{Remark}

\numberwithin{equation}{section}

\usepackage{accents}

\begin{document}

\title[ Nonlinear Generalized Collatz-Wielandt formula]{On a  generalized Collatz-Wielandt formula\\ and finding saddle-node bifurcations}


\author{Yavdat~Ilyasov}
\address{Institute of Mathematics of UFRC RAS, 112, Chernyshevsky str., 450008, Ufa,
Russia}
\curraddr{Instituto de Matem\'atica e Estat\'istica. Universidade Federal de Goi\'as,
74001-970, Goiania, Brazil}
\email{ilyasov02@gmail.com}


\subjclass[2000]{Primary 35B32, 49J52; Secondary 49R05, 65H10, 65P30}

\date{}


\keywords{ Saddle-node bifurcation,  Collatz-Wieland formula, Nonlinear system of equations}

\begin{abstract}
The Collatz-Wielandt   formula obtained by Lothar Collatz  (1942)  and  Helmut Wielandt (1950) provides a simple variational characterization of the Perron–Frobenius eigenvalue of certain types of matrices. 
In the present paper, we  substantiate the  nonlinear generalized Collatz-Wielandt  formula 
$$
\lambda^*= \sup_{x\in Q}\min_{i:h_i(x) \neq 0} \frac{g_i(x)}{ h_i(x)}, ~~Q \subset \mathbb{R}^n,
$$
and prove that its solution $(x^*,\lambda^*)$ yields the maximal saddle-node bifurcation for  systems of equations of the form:  
$
g_i(x)-\lambda  h_i(x)=0, ~~x \in Q.
$
Using this we introduce a  simply verifiable criterion for the detection of saddle-node bifurcations of a given system of equations. 
 We apply this criterion to prove the existence of the maximal saddle-node bifurcations for finite-difference approximations of nonlinear partial differential equations and for the system of power flow equations. 

\end{abstract}

\maketitle

\section{Introduction}

This paper is concerned with finding  bifurcations for the system of equations of the form:
	\begin{align}
	\tag{$f$}
	\label{f}
	f(x,\lambda):=g(x)-\lambda  h(x)=0, ~~~~x \in Q.
	\end{align}
	Here  $Q$ is an open  domain in $\mathbb{R}^n$,  $\lambda$  is a real parameter, $g, h: \mathbb{R}^n \to \mathbb{R}^n$ are continuously differentiable vector functions and we suppose that $h(x)\neq 0$, $h(x) \geq 0$  in $Q$. 	By a solution of system \eqref{f}, we mean a pair $(x, \lambda)$ which satisfies \eqref{f}.
	%


	A solution $(x^*,\lambda^*)$ of \eqref{f} is called   \textit{saddle-node  bifurcation} (or \textit{fold bifurcation},\textit{ turning point})   of \eqref{f} in $Q$ if there exists a $C^1$-map $(-a,a) \ni s \longmapsto (x(s),\lambda(s)) \in Q\times \mathbb{R}$
	for some $a>0$ such that 
	%
		%
		\begin{itemize}
			\item \textit{$(x(s),\lambda(s))$ satisfies  \eqref{f} for $ s \in (-a,a)$ and }$(x(0),\lambda(0))=(x^*,\lambda^*)$,
		\item ${\displaystyle \frac{d}{ds}\lambda(0)=0}$ and $
	\lambda(s) \in (-\infty,\lambda^*]$\textit{ or }$\lambda(s) \in  [\lambda^*,+\infty)$, $ \forall s \in (-a,a)$.

\end{itemize}


Bifurcation phenomena arise in many fields of science and technology, including quasi-classical and quantum physics, general relativity, analysis of voltage stability of power systems, neural networks, biology, ecology, and many others (see, e.g., \cite{Ajjar, keller1, kohli,  seydel}).

This paper aims to address the following problems:
\begin{itemize}
	\item  \textit{Finding general sufficient conditions under which a given system has a saddle-node bifurcation.}
	\item \textit{Finding an explicit variational formula for determining saddle-node bifurcations.}
\end{itemize}
Note that a spectral point $(\phi^*, \lambda^*)$ of a square matrix $A$ with $\lambda^* \in \mathbb{R}$ can be considered as a saddle-node bifurcation. Indeed, the map $x(s)=s \phi^*, \lambda(s)=\lambda^*$, $s \in \mathbb{R}$ satisfies all the above conditions  of the saddle-node bifurcation.
We emphasize that the spectral theory of linear operators offers a number of well-developed methods for solving the above problems. Moreover, the spectral theory provides constructive methods of finding spectral points, including the method of the characteristic polynomial,  the  Courant–Fischer–Weyl and Collatz–Wielandt variational principles.

For nonlinear problems, however, the situation is more complicated.  
The existing, well-known approaches, such as the Crandall–Rabinowitz \cite{grand}, Krasnoselskii \cite{krasn} and Rabinowitz \cite{rabin} theorems or the Vainberg–Trenogin branching method  \cite{Vainb}, etc., allow to prove the existence of bifurcations for nonlinear problems but omit the main question of how to find a point $(x^*,\lambda^*)$ satisfying the necessary saddle-node bifurcation properties (see \cite{keller1, kielh, kuzn}).

%

%

%

 %
%

The purpose of this article is to present a method which allows to obtain a complete answer to the question of the existence of  saddle-node bifurcations for given  systems of equations. Moreover, we are aimed to get a formula for determining bifurcations that would be useful in the numerical calculation of saddle-node bifurcations as well. In addition to this, we suppose  to provide  a contribution to the development of the Perron–Frobenius theory to the nonlinear problems.

Our approach is based on the extended functional method introduced in \cite{IlyasFunc}, however, the present paper is kept self-contained and full proofs are provided.

%
%


%

Let us state our main results. We shall look for a saddle-node bifurcation of \eqref{f} by means of the following  \textit{nonlinear generalized Collatz-Wieland formula} \cite{IlyasFunc}
	\begin{equation}\label{MiMaSI}
	\lambda^*=\max_{x\in Q}\lambda(x)\equiv \max_{x\in Q}\min_{i:h_i(x) \neq 0} \frac{g_i(x)}{ h_i(x)}.
	\end{equation}
Here $\lambda(x):= \min_{i:h_i(x) \neq 0}  \frac{g_i(x)}{ h_i(x)}$, $x\in Q$, is called the \textit{ functional of bifurcations}.  
Denote by $J_x f(x, \lambda)=(\frac{\partial f}{\partial x_i})_{1\leq i \leq n}$ the Jacobian matrix of $f(x,\lambda)$. 

We assume  the following condition:

\medskip

\begin{description}
	\item[(R)] \textit{ $J_x f(x,\lambda)$  is  irreducible off-diagonal sign-constant matrix for all $ x \in  Q$ and $\lambda \in \mathbb{R}$}.
\end{description}

\medskip

Hereafter, a real square matrix $A$ is called \textit{off-diagonal sign-constant} if one of the following is true: $a_{ij}\geq 0$ or, $a_{ij}\leq 0$ for all $i,j \in \{1,\ldots,n\}$, $i\neq j$. The matrix $A$ is irreducible if it cannot be conjugated into block upper triangular form by a permutation matrix $P$ (see  \cite{Varga}).

Our main result is as follows
\begin{theorem}\label{thm1}  Assume that {\rm\textbf{(R)}} holds true and $h(x)\neq 0$, $h(x) \geq 0$ in $Q$. Suppose  $\lambda^*<+\infty$. 
\begin{description}
	\item[$(1^o)$] Then \eqref{f} has no solutions in $Q$ for any $\lambda >\lambda^*$.
	\item[$(2^o)$] If there exists  a  maximizer  $x^* \in Q$ of $\lambda(x)$  such that $h_i(x^*)\neq  0$, $\forall i=1,\ldots,n$, then $(x^*,\lambda^*)$ is a  maximal saddle-node bifurcation point  of \eqref{f} in $Q$. 	Moreover,  ${\rm dim\, Ker}(J_x f(x^*,\lambda^*))=1$ and both right and left eigenvectors of $J_x f(x^*,\lambda^*)$ are  strongly positive.
\end{description}
\end{theorem}
We call $ \zeta \in {\rm Ker} J_x f(x,\lambda)$ and $ \xi \in {\rm Ker}(J_x f(x,\lambda))^T$ the \textit{right and left eigenvectors} of  $J_x f(x,\lambda)$. A saddle-node bifurcation  $(x^*,\lambda^*)$ of \eqref{f} is said to be  \textit{maximal}  in $Q$ if  $\bar{\lambda}\leq  \lambda^*$  for any other saddle-node bifurcation  $(\bar{x},\bar{\lambda})$ of \eqref{f} in  $Q$.

It is important to note that variational formula \eqref{MiMaSI} allows finding the saddle-node bifurcations numerically. Such investigations with related numerical experiments were initiated in \cite{IlIvan2,IlIvan1}. Notice that problem \eqref{MiMaSI} belongs to a class of nonsmooth optimization problems.  The theory of nonsmooth optimization has been intensively developed over the past few decades, and, at present, there are various powerful numerical methods for solving such problems (see, e.g., \cite{bagir1,  Demyan, DemyanSPP} and references therein). 
Thus,  along with the well-known methods for numerically finding bifurcations (\cite{keller1, kuzn, seydel}), we can  use the full range of nonsmooth optimization methods. 




\begin{remark}
To the best of our knowledge, the results on the existence of saddle-node bifurcations for finite-difference approximations of nonlinear boundary value problems proved in Section 4 have not been known before.
\end{remark}

\begin{remark}
As noted above, the nonlinear generalized Collatz-Wieland formula \eqref{MiMaSI} is a finite-dimensional  version of an infinite-dimensional minimax principle \cite{IlyasFunc}. This minimax principle has been used to solve various theoretical problems from  the nonlinear partial differential equations  \cite{BobkovIlyasov, ilcras} including problems which are not directly related to the finding of bifurcations (see e.g. \cite{IlD,IlRunst}).
\end{remark}

The paper is organized as follows. In Section 2, we give some
preliminaries.  Theorems \ref{thm1} is proved in Sections 3. In Section 4, we present some examples of applications of Theorem \ref{thm1}.

\section{Preliminaries}

		Hereafter, $\|\cdot\|$ and $\left\langle \cdot ,\cdot \right\rangle$ stand for the Euclidean norm and the scalar product in $\mathbb{R}^n$, respectively;    ${ \nabla_x=(\frac{\partial }{\partial x_1} \cdots \frac{\partial }{\partial x_n})^T}$  denotes the gradient. 	
Furthermore, we write $h(x)\neq 0$ in $Q$ if $ \forall x \in Q$, $\exists i \in \{1,...,n\}$ s.t. $h_i(x)\neq 0$. Denote $\accentset{\circ}{Q}= \bigcap_{i=1}^n \{x\in Q: ~h_i(x)\neq  0\}$.
In what follows, we say that a solution  $(\bar{x},\bar{\lambda})$ of \eqref{f} is \textit{maximal}  in $Q$ if \eqref{f}  has no solution $(x,\lambda)$ such that $\lambda>\bar{\lambda}$ and $x\in Q$.

The proof of the next lemma can be obtained from many sources (see, e.g.,  \cite{grand,keller1,kielh, Vainb}). 
\begin{lemma}\label{lemkell}
Assume that $Q$ is an open domain in $\mathbb{R}^n$.
Suppose that  $(x^*, \lambda^*)$ is a maximal  solution of \eqref{f} in $Q$ such that 
$$
{\rm dimKer}(J_x f(x^*, \lambda^*))=1~~\mbox{and}~~
	\frac{\partial}{\partial \lambda} f(x^*, \lambda^*) \not\in R(J_x f(x^*, \lambda^*))
$$
Then $(x^*,\lambda^*)$ is a maximal saddle-node bifurcation point of \eqref{f} in $Q$.
\end{lemma}

Let us introduce 
\begin{equation*}\label{ND}
	N(x)=\{i\in \{1, \ldots,n\}:~r_i(x)=\lambda(x), ~h_i(x)\neq 0\},~~x \in Q,
\end{equation*} 
where $\displaystyle{r_i(x):= \frac{g_i(x)}{ h_i(x)}}$ for $x \in Q$ s.t. $h_i(x)\neq 0$, $i=1, \ldots,n$. Denote by $|N(x)|$  the number of elements in $N(x)$. 
\begin{remark}\label{rem:1}
	For $x \in Q$, the condition	$|N(x)|=n$ (i.e., $N(x)=\{1, \ldots,n\}$) is satified  if and only if $(x, \lambda(x))$ is a solution of \eqref{f} and $~h_i(x)\neq 0$, $\forall i \in \{1, \ldots,n\}$.
\end{remark}

Introduce the following convex hull 
	\begin{align*}\label{CH1}
	\partial \lambda(x):=\{z=\sum_{i \in N(x)}\zeta_i\nabla_x r_i(x):~\sum_{i \in N(x)}\zeta_i=1,~ \zeta_i\geq 0,~ i\in N(x)\}.
	\end{align*}
\begin{definition}\label{Def:St}
		 A point $\hat{x} \in Q$ is called \textit{stationary point} of $\lambda(x)$ if   $0 \in \partial \lambda(\hat{x})$.
\end{definition}
For $h_i(x)\neq 0$ we have
\begin{equation}\label{diffr}
	\nabla_x r_i(x)=\frac{1}{ h_i(x)} (\nabla_x g_i(x)-r_i(x)\nabla_x h_i(x) )=\frac{1}{ h_i(x)} \nabla_x f_i(x,\lambda)|_{\lambda=r_i(x)},
\end{equation}
and thus, for  stationary point $\hat{x}$ of $\lambda(x)$ we have
$$
\sum_{i \in N(\hat{x})} \zeta_i\nabla_x r_i(\hat{x})|_{\lambda=\lambda(\hat{x})}=(J_x f(\hat{x},\lambda(\hat{x})))^T\hat{\xi}=0,
$$
where $\hat{\xi}_i=\zeta_i/h_i(\hat{x})$ for $i\in N(\hat{x})$, $\hat{\xi}_i=0$ for $i \in \{1, \ldots,n\}\setminus N(\hat{x})$ and $\sum_{i \in N(\hat{x})}\zeta_i=1,~ \zeta_i\geq 0,~ i\in N(\hat{x})$.

A point  $\hat{x} \in Q$ is called local maximizer of $\lambda(x)$ in $Q$ if there exists a neighbourhood $ U(\hat{x}) \subset Q$ of $\hat{x}$ such that  $\lambda(\hat{x}) \leq \lambda(x)$ for any $x \in  U(\hat{x})$.

\begin{lemma}\label{DemNeceAppen} 
Let $\hat{x}$ be a local maximizer of the bifurcation functional $\lambda(x)$ in $Q$ and $\hat{x} \in \accentset{\circ}{Q}$. Then there exist real numbers $\zeta_i$, $i=1,\ldots, n$ such that   $\zeta_i\geq 0$, for $i\in N(\hat{x})$ and $\zeta_i= 0$, for $i\in \{1,\ldots,n\}\setminus N(\hat{x})$, $\sum_{i \in N(\hat{x})} \zeta_i =1$ and $$\sum_{i \in N(\hat{x})}\zeta_i \nabla_x r_i(\hat{x})= 0.$$
\end{lemma} 
\begin{proof} Let $\hat{x} \in \accentset{\circ}{Q}$ be a local maximizer of $\lambda(x)$. Then 
$$
\hat{\lambda}=\lambda(\hat{x})=\max_{x\in U(\hat{x})}\lambda(x)=\max_{x\in U(\hat{x})}\min_{i} r_i(x).
$$
for some neighbourhood $ U(\hat{x}) \subset \accentset{\circ}{Q}$.
Then $(\hat{\lambda}, \hat{x})$ is a maximizer of the following constraint maximization problem
\begin{equation}\label{EqMIM}
	\begin{cases}
	&\text {maximize} ~~\, \lambda,\\
	& \text {subject to:}~~ r_i(x)\geq \lambda, ~i=1,\ldots,n,\\
		&~~~~~~~~~~~~~~~~~x \in  U(\hat{x}),~\lambda \in \mathbb{R}.
\end{cases}
\end{equation}
Hence, the Lagrange multiplier rule (see Theorem 48.B in \cite{zeidler}) implies that there exist $\mu_0, \mu_1, \ldots, \mu_n \in \mathbb{R}$ such that $\sum_i |\mu_i|\neq 0$ and 
\begin{align}
  &\mu_0\frac{d}{d\lambda} \lambda+\sum_{i=1}^n \mu_i\frac{d}{d\lambda}(r_i(\hat{x})- \lambda)=0~\Leftrightarrow ~\mu_0-\sum_{i=1}^n \mu_i=0,\label{Lag1}\\
	&\sum_{i=1}^n\mu_i\nabla_x r_i(\hat{x})= 0,\label{Lag}\\
	&\mu_i\leq 0,~~ i=0,\ldots,n,\label{Lag3}\\
	&\mu_i(r_i(\hat{x})- \lambda)=0, ~~i=1,\ldots,n.\label{CS}
\end{align}
From \eqref{Lag1} and \eqref{Lag3} it follows that $\mu_0\neq 0$. Thus, we may assume that $\mu_0=-1$  and consequently, $\sum_{i=1}^n \mu_i=-1$.
Set $\zeta_i=-\mu_i$, $i=1,\ldots, n$. Then $\zeta_i \geq 0$, $i=1,\ldots, n$ and by \eqref{CS}, $\zeta_i= 0$ for any $i\in \{1,\ldots,n\}\setminus N(\hat{x})$. Finally, taking into account \eqref{Lag1}, \eqref{Lag} we conclude the proof of Lemma \ref{DemNeceAppen}.  
\end{proof}



\section{Proof of Theorems \ref{thm1}}

$(1^o)$ Suppose $\lambda^*<+\infty$. Let us show that \eqref{f} has no solutions in $Q$ for any $\lambda >\lambda^*$.  Conversely, suppose that there exists a solution $(x_\lambda,\lambda)$ of \eqref{f} in $Q$ with $\lambda>\lambda^*$. Then  \eqref{MiMaSI} implies  $\min_{i:h_i(x) \neq 0} r_i(x_\lambda)\leq \lambda^*<\lambda$. Hence, there exists $i_0 \in \{1, \ldots,n\}$ such that $ r_{i_0}(x_\lambda)< \lambda$ or, equivalently $g_{i_0}(x_\lambda)-\lambda h_{i_0}(x_\lambda)<0$ which is a contradiction.

$(2^o)$ Let $x^*\in Q$ be a maximizer in \eqref{MiMaSI}  such that $x^* \in \accentset{\circ}{Q}$. Then Lemma \ref{DemNeceAppen}  implies that there exist real numbers $\zeta_i$, $i\in N(x^*)$ such that  $\zeta_i\geq 0$ for $i\in N(x^*)$,
 $\sum_{i} \zeta_i =1$ and 
$$
\sum_{i \in N(x^*)}\zeta_i \nabla_x r_i(x^*)=(J_x f(x^*,\lambda(x^*)))^T\xi^*=0,
$$
where we set  $\xi^*_i=\zeta_i/h_i(x^*)$ for $i\in N(x^*)$ and $\xi^*_i=0$ for $i \in \{1, \ldots, n\}\setminus N(x^*)$. 
Due to assumption {\rm\textbf{(R)}}, this is possible only if $|N(x^*)|=n$ and $\xi^*_i>0$ for all $i=1,\ldots,n$. Thus $ \xi^* \in {\rm Ker}(J_x f(x^*,\lambda^*))^T$ and the equality $|N(x^*)|=n$ implies that  $(x^*,\lambda^*)$ is a solution of \eqref{f}. Now taking into account that \eqref{f} has no solutions in $Q$ for  $\lambda >\lambda^*$ we conclude that 
$(x^*,\lambda^*)$  is a  maximal solution  of \eqref{f} in $Q$.

It is known that any irreducible off-diagonal sign-constant matrix  has a unique strictly
positive eigenvector, with a real simple eigenvalue  \cite{Varga}. Hence, by {\rm\textbf{(R)}},  ${\rm dim Ker} J_x f(x^*,\lambda^*)={\rm dim Ker}(J_x f(x^*,\lambda^*))^T=1$ and  the right  and left eigenvectors $\zeta^*$, $\xi^*$ of $J_x f(x^*,\lambda^*)$ are  strongly positive. From this and since  $h(x^*)\neq 0$, $h(x^*) \geq 0$, we obtain 
$$
	\langle\frac{d}{d \lambda}f(x^*,\lambda^*),\xi^* \rangle =\left\langle h(x^*),\xi^*\right\rangle \neq 0,
	$$
	whence since ${ span}(\xi^*) ={\rm Ker}(J_x f(x^*, \lambda^*))^T$, it follows that	
	$$
	\frac{\partial}{\partial \lambda} f(x^*, \lambda^*) \not\in R(J_x f(x^*, \lambda^*)).
	$$
Thus, all the assumptions of Lemma \ref{lemkell}	are satisfied. Consequently,  $(x^*,\lambda^*)$  is a  maximal saddle-node bifurcation point  of \eqref{f} in $Q$.

\section{Examples of application}

\subsection{ Example 1.}
Consider the so-called system of power flow equations with two buses
\cite{Ajjar}
\begin{equation}\label{PF}
	\left\{
	\begin{aligned}
		&-v\sin(\theta)=\lambda p, \\
		&~~~v\cos(\theta)-v^2=\lambda q,
		\end{aligned}
	\right. 
\end{equation}
Here $\theta\in  (-\pi/2,\pi/2 )$ and $v>0$ are unknown variables, $p,q \geq 0$ are given and $\lambda \in \mathbb{R}$ is the so-called  load parameter \cite{Ajjar}. The saddle-node bifurcation of system \eqref{PF} corresponds to the so-called maximum load capacity of power system,  and its finding plays a crucial role  in the control of the voltage stability of power systems \cite{Ajjar}.

Let us show that \eqref{PF} has a saddle-node bifurcation.
Assume that $p,q>0$. Denote $Q=(-\pi/2,\pi/2 )\times \mathbb{R}^+$ and introduce
$$
r_1(\theta,v)=-\frac{v}{p}\sin(\theta),~~r_2(\theta,v)=\frac{1}{q}(v\cos(\theta)-v^2).
$$
Consider 
\begin{equation}\label{MiMPF}
	\lambda^*=\sup_{(\theta,v)\in Q}\min\{-\frac{v}{p}\sin(\theta),\frac{1}{q}(v\cos(\theta)-v^2)\}
	\end{equation}
Since $\min\{	r_1(0,1/2),r_2(0,1/2)\}=1/4q>0$, we infer that $\lambda^*\geq 1/4q>0$. Using this it is not hard to show  that there exists a maximizer $(\theta^*,v^*) \in Q$ of \eqref{MiMPF}. Moreover, $(\theta^*,v^*)$ is an internal point in $Q$, i.e., $\theta^* \neq -\pi/2,\pi/2$ and $v^* \neq 0$. Observe, the Jacobian matrix
$$
 J_{\theta,v} f(\theta,v,\lambda)={\begin{pmatrix}-\frac{v}{p}\cos(\theta) &~-\frac{1}{p}\sin(\theta)\\-\frac{v}{q}\sin(\theta)&\frac{1}{q}(\cos(\theta)-2v)\end{pmatrix}}
$$
satisfies condition \textbf{(R)}.	Hence applying Theorem \ref{thm1} we deduce 
\begin{lemma} Assume $p,q>0$. 
Then $\lambda^*<+\infty$ and there exists a maximizer $(\theta^*,v^*)$ of \eqref{MiMPF}. Furthermore, 
	$(\theta^*,v^*, \lambda^*)$ is a  maximal saddle-node bifurcation point  of \eqref{PF}  in $(-\pi/2,\pi/2 )\times \mathbb{R}^+$. 	Moreover,  ${\rm dim\, Ker}(J_{\theta,v} f(\theta^*,v^*, \lambda^*))=1$ and both right and left eigenvectors of $J_{\theta,v} f(\theta^*,v^*, \lambda^*)$ are  strongly positive.
\end{lemma}
\begin{remark}
	For this simple case  a similar result can be obtained directly by solving  system  \eqref{PF} (see, e.g., \cite{Ajjar}, Sec.1.3.1). However, we are not sure that this approach is simpler than what is proposed above (cf. \cite{Ajjar}). Moreover, we conjecture that our approach can be generalized to systems of power flow equations of large dimensions. 
\end{remark}

\subsection{ Example 2.} 
 Consider the boundary value  problem with the so-called nonlinearity of convex-concave type
\begin{equation}\label{p}
	\left\{
	\begin{aligned}
		&-\Delta u = \lambda u^{q}+p(u)~~\mbox{in}~~\Omega, \\
		&~u\geq 0~~\mbox{in}~~\Omega,\\
		&~~~~ u=0~~\mbox{on}~~\partial\Omega,
	\end{aligned}
	\right. 
\end{equation}
where $\Omega$ is a bounded smooth domain in $\mathbb{R}^N$, $N \geq 1$.
Assume that $0<q<1$, $p \in C^1(\mathbb{R})$ and
\begin{description}
	\item[(H)] $\displaystyle{\lim_{t\to +\infty} \frac{p(t)}{t^q} =+\infty}$,~~ $\displaystyle{\lim_{t\to 0} \frac{p(t)}{t} =0}$.
\end{description}
An example of the function $ p $ is as follows $p(s)=s^\gamma$, $s \in \mathbb{R}^+$ where $1<\gamma<+\infty$. Thus, in this case, the nonlinearity  $ \lambda u^{q}+u^\gamma$ is true convex-concave.

We  use finite differences to approximate this problem and for the sake of simplicity, we restrict ourselves to the case $N=1$.    

Assume that $\Omega=(0,L)$, $L>0$. Set  $x_i=i \cdot \tau$, $u_i=u(x_i)$,  $1\leq i\leq n$, where $\tau=L/(n+1)$. For the second derivatives at $n$ mesh points we used  a standard second-order finite difference approximation. This yields for \eqref{p} the following approximating system of $n$ nonlinear  equations
\begin{equation}
\label{dis1} 
	\begin{cases}
	-\frac{u_{i+1}-2 u_i +u_{i-1}}{\tau^2} - p(u_i)- \lambda u_i^q=0, \\
	~ 1\leq  i\leq n,
	\end{cases}
\end{equation}
where $u=(u_1, \ldots, u_n) \in S:=\{u \in \mathbb{R}^n : ~u_i > 0, ~i=1,\ldots,n\}$, $u_0=u_{n+1}=0$.

Then we have
\begin{eqnarray*}
&& f_1(u, \lambda)=-\frac{u_{2}-2 u_1 }{\tau^2 }- p(u_1) - \lambda u_1^q, \\
&& f_n(u, \lambda)=-\frac{-2 u_n +u_{n-1}}{\tau^2 } - p(u_n) - \lambda u_n^q,\\
&& f_i(u, \lambda)=-\frac{u_{i+1}-2 u_i +u_{i-1}}{\tau^2 }- p(u_i)  - \lambda u_i^q,~~i=2,...,n-1,
\end{eqnarray*}
Introduce
\begin{eqnarray*}
&& r_1(u)=\frac{-u_{2}+2 u_1 -\tau^2 p(u_1)}{\tau^2 u_1^q},~
 r_n(u)=\frac{2 u_n -u_{n-1}-\tau^2 p(u_n)}{\tau^2 u_n^q},	\\
&& r_i(u)=\frac{-u_{i+1}+2 u_i -u_{i-1}-\tau^2 p(u_i)}{\tau^2 u_i^q},~~i=2,...,n-1.
\end{eqnarray*}
The nonlinear generalize Collatz-Wieland formula for \eqref{dis1}  is as follows:
\begin{equation}\label{MiMCC}
	\lambda^*=\sup_{u\in S}\lambda(u)=\sup_{u\in S}\min_{i:u_i\neq 0} r_i(u),
	\end{equation}
	
\begin{lemma}\label{L-CC} Assume that {\rm \textbf{(H)}} holds true. Then $\lambda^*<+\infty$ and there exists a maximizer $u^* \in S$ of \eqref{MiMCC}. Furthermore, 
	$(u^*,\lambda^*)$ is a  maximal saddle-node bifurcation point  of \eqref{dis1} in $S$. 	Moreover, ${\rm dim\, Ker}(J_u f(u^*,\lambda^*))=1$ and both right and left eigenvectors of $J_u f(u^*,\lambda^*)$ are  strongly positive.
\end{lemma}
\begin{proof} To prove the lemma it is sufficient to verify for \eqref{dis1} that all the assumptions of  Theorem \ref{thm1} are satisfied. 

Let us show  that $\lambda^*>0$. Take $u(\delta)=(u_1(\delta), \ldots, u_n(\delta))$ such that $u_1(\delta)=\delta>0$, $u_j(\delta)=0$, $j=2,\ldots, n$. Then $u(\delta) \in \overline{S} \setminus 0$ and
$$
\lambda^*\geq \lambda(u(\delta))=\min_{i:u_i\neq 0}r_i(u(\delta))=r_1(\delta)=\delta^{1-q}(\frac{2}{\tau^2}-\frac{p(\delta)}{\delta})>0
$$
for sufficiently small $\delta$, since assumption \textbf{(H)} implies that $\displaystyle{\frac{p(\delta)}{\delta} \to 0}$ as $\delta \to 0$.

For  $u=tv$, $t=\|u\|$, $\|v\|=1$, we have 
\begin{equation}\label{Z}
	\left\{
\begin{aligned}
	& r_1(tv)=t^{1-q}\frac{-v_{2}+2 v_1 }{\tau^2 (v_1)^q}-\frac{p(t v_1) }{\tau^2 (tv_1)^q} ,\\
	& r_n(tv)=t^{1-q}\frac{2 v_n -v_{n-1}}{\tau^2 (v_n)^q}-\frac{p(t v_n) }{\tau^2 (t v_n)^q},	\\
	& r_i(tv)=t^{1-q}\frac{-v_{i+1}+2 v_i -v_{i-1}}{\tau^2 (v_i)^q}-\frac{p(t v_i) }{\tau^2 (t v_i)^q},~~i=2,...,n-1,	
	\end{aligned}
	\right.
\end{equation}
Let $(u^k)$ be a maximizer sequence of \eqref{MiMCC}, i.e., $\lambda(u^k) \to \lambda^*$ as $k \to \infty$. Suppose that   $||u^k||:=t^k \to \infty$.
 Then there exists a subsequence,  again denoted by $(u^k)$, such that 
$ \lambda(u^k)\to -\infty$ as $k \to \infty$. Indeed, since $||v^k||=1$, there exists $i \in \{1,\ldots,n\}$  and a subsequence $v^{k_j}_i$ such that $v^{k_j}_i \to \delta>0$ as $k_j \to +\infty$. Hence, (\textbf{H}) implies that $p(t^{k_j} v^{k_j}_i)/\tau^2 (t^{k_j} v^{k_j}_i)^q \to -\infty$ as $k_j \to +\infty$ and consequently, by \eqref{Z} we have  $\lambda(u^{k_j}):=\min_{i} r_i(u^{k_j}) \to -\infty$ as $k_j \to \infty$. However, by assumption we have $\lambda(u^{k_j}) \to \lambda^*>0$. Thus $||u^k||$ is bounded and therefore there exists a limit point $u^* \in \overline{S}$ of $(u^k)$.  Passing to a subsequence if it's necessary, we may assume that $u^*=\lim_{k\to \infty}u^k$. We claim that $u^*\neq 0$.
Indeed, assume  $t^k:=||u^k|| \to 0$. First, let us suppose that there exists $\sigma>0$ such that $|v^k_i|>\sigma>0$, $\forall i=1,\ldots,n$ and  $k=1,\ldots $. Then from \eqref{Z} it follows that  $\lambda(u^{k}) \to 0$ as $k \to \infty$ which contradictions to $\lambda^*>0$. In the case $v^k_i \to 0$ for some $i\in \{2,\ldots,n-1\}$, we have
$$
r_i(t^kv^k)\leq (t^k)^{1-q}\frac{2 v^k_i }{\tau^2 (v^k_i)^q} \to 0,
$$
and similarly, $\lim_{k \to \infty}r_i(t^kv^k)\leq 0$ if $v^k_i \to 0$ for $i=1,n$. Thus, we get a contradiction  and $u^*\neq 0$.  In the same manner we can see that $u^* \in S$. 


Observe that the Jacobian matrix of $f(u,\lambda)$ has the following tridiagonal form:
$$
 J_u f(u,\lambda)={\begin{pmatrix}\frac{\partial}{\partial u_1}f_1(u,\lambda) &~-\frac{1}{\tau^2}\\-\frac{1}{\tau^2}&\frac{\partial}{\partial u_2}f_2(u,\lambda)&~-\frac{1}{\tau^2}\\&-\frac{1}{\tau^2}&\ddots &\ddots \\&&\ddots &\ddots &~-\frac{1}{\tau^2}\\&&&-\frac{1}{\tau^2}&\frac{\partial}{\partial u_n}f_n(u,\lambda)\end{pmatrix}}.
$$
Hence, we see that  $J_u f(u,\lambda)$ is a irreducible  off-diagonal sign-constant  matrix for any $u \in S$ and $\lambda \in \mathbb{R}$. Thus  condition \textbf{(R)} is satisfied for \eqref{dis1}.	
 Hence we see that all the assumptions of Theorem \ref{thm1} are satisfied and this completes the proof of the lemma.
\end{proof}

\subsection{ Example 3.} 
 Consider the Liouville-Bratu-Gelfand problem \cite{bratu, gelfand, Joseph}
\begin{equation}\label{BG}
	\left\{
	\begin{aligned}
		&-\Delta u = \lambda e^u~~\mbox{in}~~\Omega, \\
		&~~~~ u=0~~\mbox{on}~~\partial\Omega.
	\end{aligned}
	\right. 
\end{equation}
For the sake of simplicity, we restrict ourselves to the case $N=1$ so that we take $\Omega=(0,L)$, $L>0$. Set  $x_i=i \cdot \tau$, $u_i=u(x_i)$,  $1\leq i\leq n$, where $\tau=L/(n+1)$. Using  a standard second-order finite difference approximation we derive for \eqref{BG} the following approximating system of $n$ nonlinear  equations
\begin{equation}
\label{dis2} 
	\begin{cases}
	-\frac{u_{i+1}-2 u_i +u_{i-1}}{\tau^2} -  \lambda e^{u_i}=0, \\
	~ 1\leq i\leq n,
	\end{cases}~
\end{equation}
where $u=(u_1, \ldots, u_n) \in S$ and $u_0=u_{n+1}=0$. 

The nonlinear generalize Collatz-Wieland formula now reads as follows:
\begin{equation}\label{MiMBG}
	\lambda^*=\sup_{u\in S}\lambda(u)=\sup_{u\in S}\min_{i:u_i\neq 0} r_i(u),
	\end{equation}
	where
	\begin{eqnarray*}
&& r_1(u)=\frac{-u_{2}+2 u_1 }{\tau^2  e^{u_1}},~~~
 r_n(u)=\frac{2 u_n -u_{n-1}}{\tau^2  e^{u_n}},	\\
&& r_i(u)=\frac{-u_{i+1}+2 u_i -u_{i-1}}{\tau^2  e^{u_i}},~~i=2,...,n-1.
\end{eqnarray*}
\begin{lemma}\label{L-BG} $\lambda^*<+\infty$ and there exists a maximizer $u^*$ of \eqref{MiMBG}. Furthermore, 
	$(u^*,\lambda^*)$ is a  maximal saddle-node bifurcation point  of \eqref{dis2} in $S$. 	Moreover, there holds ${\rm dim\, Ker}(J_u f(u^*,\lambda^*))=1$ and both right and left eigenvectors of $J_u f(u^*,\lambda^*)$ are  strongly positive.
\end{lemma}
\begin{proof}
Take $\bar{v}=(v_1,\ldots,v_n)$ so that $v_i=sin(i \cdot \tau)$, $i=1,...,n$. Then $\bar{v} \in S$ and  $\lambda^*\geq \lambda(\bar{v})=\min_{i=1}^nr_i(\bar{v})>0$.

Let $(u^k)$ be a maximizer sequence of \eqref{MiMBG}, i.e., $\lambda(u^k) \to \lambda^*$ as $k \to \infty$. Suppose that   $||u^k||:=t^k \to \infty$. Then we see at once that $ \lambda(u^k)\to 0$ as $k \to \infty$ which contradicts to the strong inequality $\lambda^*>0$. Thus, $||u^k||$ is bounded and there exists a limit point $u^* \in \overline{S}$ of $(u^k)$ and by continuity,  $\lambda^*=\lambda(u^*)$. Observe that if $u_i^*=0$ for some $i =1,\ldots,n$, then $r_i(u^*)\leq 0$, which implies a contradiction. Thus, $u_i^*>0$ for every $i =1,\ldots,n$.

Similar to the proof of Lemma \ref{L-CC}, one can check that the Jacobian matrix $J_u f(u,\lambda)$ has the  tridiagonal form and satisfies condition \textbf{(R)}. Hence and since $u^*\in S$, the proof of the lemma follows from Theorem \ref{thm1}.
	
\end{proof}

\bibliographystyle{amsplain}

\end{document}